\documentclass[pdflatex,sn-mathphys-num]{sn-jnl}% Math and Physical Sciences Numbered Reference Style
\usepackage{dirtytalk}
\usepackage{graphicx}%
\usepackage{multirow}%
\usepackage{xcolor}
\usepackage{amsmath,amssymb,amsfonts}%
\usepackage{amsthm}%
\usepackage{mathabx}
\usepackage{mathrsfs,hyperref}%
\usepackage[title]{appendix}%
\usepackage{xcolor}%
\usepackage{textcomp}%
\usepackage{manyfoot}%
\usepackage{booktabs}%
\usepackage{algorithm}%
\usepackage{algorithmicx}%
\usepackage{algpseudocode}%
\usepackage{listings}%
\usepackage{amsmath,amsthm,bbold}
\usepackage{amsfonts,adjustbox,amssymb,xcolor}
\usepackage{rotating,enumitem}
\usepackage{url}
\usepackage[utf8]{inputenc}
\usepackage[english]{babel}
\usepackage{tikz}
\usepackage{pgfplots}
\pgfplotsset{compat=1.18}
\usepackage{hyperref}
\usepackage{dirtytalk}

\newtheorem{definition}{Definition}

\newcommand{\mA}{\mathcal A}

\theoremstyle{thmstyleone}%
\newtheorem{theorem}{Theorem}%  meant for continuous numbers
\newtheorem{lemma}{Lemma}
\theoremstyle{thmstyletwo}%

\theoremstyle{thmstylethree}%

\newtheorem{corollary}{Corollary}
\begin{document}

\title[Article Title]{A note on the stability of Arens products under small perturbations of multiplication.}
\author{Deepika Rajoriya}

%\author[]{\fnm{Lav} \sur{Kumar Singh}}\email{lav.singh@imfm.si}

%\author[]{\fnm{Aljoša} \sur{Peperko}}\email{aljosa.peperko@fs.uni-lj.si}

%\affil[]{\begin{center}\orgdiv{Department of Mathematics}, \orgname{Institute of Mathematics, Physics and Mechanics}, \orgaddress{\street{Jadranska Ulica, 19}, \city{Ljubljana}, \postcode{1000}, \country{Slovenia}}\end{center}}
\abstract{ In this short note, we study the small perturbations of multiplication structure of a Banach algebra $\mathcal A$, and its impact on the two Arens products on the bidual $\mathcal A^{**}$. We show that the property of being \say{Arens irregular} is stable under sufficiently small perturbation in the multiplication structure of a Banach algebra. Further we show that for any Banach function algebra which is equivalent to a uniform algebra, the Arens regularity is stable under small perturbation on the multiplication structure of the original Banach function algebra.  }




\maketitle
\section{Introduction}
Stability of a property under a small perturbation is studied in various disciplines like physics, numerical analysis, differential equations, functional analysis etc. Small perturbations of multiplication in a Banach algebra provide significant insight into the structure complexity of the Banach algebra. Impact of small perturbations of multiplication on various properties has been studied by sevaral authors in \cite{Dosi, Jarosz1985, Johnsons}. In this note we study the effect of small perturbations of multiplication structure of a Banach algebra on its second dual algebra (with respect to both the Arens products). Consequently, we establish the fact that being Arens irregular is stable under small perturbations [Cor. \ref{cor1}]. Further we show that for Banach function algebras which are equivalent to a uniform algebra, the Arens regularity is also stable [Th. \ref{BFAU}]. Such algebras are said to satify "Stable Arens regularity (SAR)" property.

\subsection{Arens products}
We quickly recall the definitions of the two products $\Box$ and
$\Diamond$.  Let $\mA$ be a Banach algebra. For $a \in \mA$,
$\omega\in \mA^*$, $f\in \mA^{**}$, consider the functionals
$\omega_a, {}_{a}\omega\in \mA^*$, $\omega_f,{}_f\omega\in \mA^{**}$
given by
\[
w_a=(L_a)^*(\omega),~~
{}_a\omega=(R_a)^*(\omega);~~ \omega_f(a)=f({}_a\omega)\text{ and }
{}_f\omega(a)=f(\omega_a).
\]
Then, for $f,g\in \mA^{**}$ the
operations $\square$ and $\diamond$ are given by
\[
(f\square g)
(\omega)= f({}_g\omega)\text{ and } (f\diamond
g)(\omega)=g(\omega_f)
\]
for all $\omega \in \mA^*$. And, $\mA$ is said to be {\em Arens regular} if the
products $\Box$ and $ \Diamond$ agree. Else the Banach algebra is said to be {\em Arens irregular}. Equivalently, Arens products can be characterized by iterated weak-* limits
\begin{align*}
	f\square g&=\overset{j}{\underset{w^*}{\lim }}\overset{i}{\underset{w^*}{\lim}}~ f_j g_i\\
	f\diamond g&=\overset{i}{\underset{w^*}{\lim }}\overset{j}{\underset{w^*}{\lim}}~ f_jg_i
\end{align*}
where $\{f_j\}$ and $\{g_i\}$ are bounded nets  in $\mathcal A$ converging to $F$ and $G$ respectively in the weak-* topology on $\mathcal A^{**}$. Arens regularity has been a central theme of study in the domain of Banach algebras for quite some time (see for instance \cite{Lav0} and \cite{Lav1} for recent developements). They also help us realize the second dual of a $C^*$-algebra as a von Neumann algebra (equipped with either products) \cite{Lav2}.  \\

 \section{Small perturbations of Banach algebras}
\begin{definition} \cite{Jarosz1985} An (algebraic) $\epsilon$-perturbation of a Banach algebra $(\mathcal A, \boldsymbol{\cdot},\|.\|)$ is  the algebra $(\mathcal A, \times)$ such that $$\|a\times b-a\cdot b\|\leq \epsilon\|a\|.\|b\|$$ for all $a,b\in \mathcal A$.\end{definition}  It is easy to see that the new multiplication $\times$ is continuous with respect to the original norm $\|.\|$, but this norm may not be sub-multiplicative with respect $\times$. Using the continuity of $\times$, one can always define an equivalent norm $\|.\|_1$ on $\mathcal A$ such that $(\mathcal A, \times, \|.\|_1)$ becomes a Banach algebra. Let the equivalence of two  norms be given by $$c_1\|a\|\leq \|a\|_1\leq c_2\|a\|$$  for some $c_1,c_2>0$. Due to this equivalence, the continuous dual of $\mathcal A$ is same vector space under both the multiplication.
 \begin{lemma}\label{perturb}
 	If $(\mathcal A, \cdot~, \|.\|)$ is a Banach algebra and $(\mathcal A, \times, \|.\|_1)$ is its $\epsilon$- perturbation then Banach algebras $(\mathcal A^{**},\square_1)$ and $(\mathcal A^{**}, \diamond_1)$ are also $\epsilon$-perturbations of $(\mathcal A^{**},\square )$ and $(\mathcal A^{**},\diamond)$ respectively. ($\square_1$ and $\diamond_1$ denote the Arens product on the bidual of $\epsilon$-perturbation.)
 \end{lemma}
 \begin{proof}
 	Let $\nu_1:\mathcal A^{**}\times \mathcal A^*\to \mathcal A^*$ denote the map $\nu_1(f,a^*)={}_f(a^*)\in (\mathcal A, \cdot,\|.\|)^*$ and $\nu_2:\mathcal A^{**}\times \mathcal A^*\to \mathcal A^*$ denote the map $\nu_2(f,a^*)={}_f(a^*)\in (\mathcal A, \times, \|.\|_1)^*$. 
 	Let $f,g\in \mathcal A^{**}$, then
 	\begin{align}\label{approx1}
 		\|f\square g-f\square_1g\|&=\sup_{\|a^*\|\leq 1}|f\square g(a^*)-f\square_1g(a^*)| \nonumber\\&=\sup_{\|a^*\|\leq 1}|f(\nu_1(g,a^*))-f(\nu_2(g,a^*))|\nonumber\\&\leq ||f||\sup_{\|a^*\|\leq 1}\|\nu_1(g,a^*)-\nu_2(g,a^*)\|
 	\end{align}
 	But 
 	\begin{align}\label{approx2}
 	\|\nu_1(g,a^*)-\nu_2(g,a^*)\|&=\sup_{\|a\|\leq 1}|\nu_1(g,a^*)(a)-\nu_2(g,a^*)(a)|\nonumber\\&=\sup_{\|a\|\leq 1}|g(\xi_1(a^*,a))-g(\xi_2(a^*,a))|\nonumber\\&\leq ||g||\sup_{\|a\|\leq 1}\|\xi_1(a^*,a)-\xi_2(a^*,a)\|
 	\end{align}
 	where $\xi_1:\mathcal A^*\times \mathcal A\to \mathcal A^*$ be the map $\xi_1(a^*,a)=(a^*)_a\in (\mathcal A,\cdot, \|.\|)^*$ and $\xi_2:\mathcal A^*\times \mathcal A\to \mathcal A^*$ denote the map $\xi_2(a^*,a)=(a^*)_a\in (\mathcal A, \times ,\|.\|_1)^*$. Then
 	 \begin{align}\label{approx3}
\|\xi_1(a^*,a)-\xi_2(a^*,a)\|&=\sup_{\|b\|\leq 1}|\xi_1(a^*,a)(b)-\xi_2(a^*,a)(b)|\nonumber\\&=\sup_{\|b\|\leq 1}|a^*(a\cdot b)-a^*(a\times b)|\nonumber\\&=\|a^*\|\sup_{\|b\|\leq 1}\|a\cdot b-a\times b\| \nonumber\\&\leq \|a^*\|\sup_{\|b\|\leq 1}\epsilon\|a\|.\|b\|\nonumber\\&=\epsilon\|a^*\|.\|a\|
 	\end{align}
 	substituting equation \ref{approx3} in equation \ref{approx2}, we get \begin{equation}\label{approx4}\|\nu_1(g,a^*)-\nu_2(g,a^*)\|\leq \epsilon\|a^*\|.\|g\|\end{equation}
 	substituting equation \ref{approx4} in equation \ref{approx1}, we get $$	\|f\square g-f\square_1g\|\leq \epsilon \|f\|.\|g\|~~~~\forall f,g\in \mathcal A^{**}.$$
 	Hence, the first Arens product on $(\mathcal A, \times ,\|.\|)^{**}$ is $\epsilon$-perturbation of the first Arens product on $(\mathcal A, \cdot, \|.\|)^{**}$. Similarly, it can be proved that the second Arens product on $(\mathcal A, \times ,\|.\|)^{**}$ is $\epsilon$-perturbation of the second Arens product on $(\mathcal A, \cdot, \|.\|)^{**}$, i.e
 	 $$	\|f\diamond g-f\diamond_1g\|\leq \epsilon \|f\|.\|g\| ~~~~\forall f,g\in \mathcal A^{**}$$
 	 \end{proof}
 \begin{corollary}\label{cor1}
 	Sufficiently small $\epsilon$-perturbations of a Arens irregular Banach algebra are Arens irregular.
 \end{corollary}
 \begin{proof}
 	Let $(\mathcal A, \cdot, \|.\|)$ be an Arens irregular Banach algebra. Then the mapping $v\mapsto u\square v$ is not weak*-weak* continuous in $\mathcal A^{**}$. Hence, there exist a net $\{v_\alpha\}_{\alpha\in \Gamma}$ converging to $v\in \mathcal A^{**}$ in weak-* topology, and an element $u\in \mathcal A^{**}$,  $a^*\in \mathcal A^*$, a $\delta>0$ and a cofinal subnet $\{v_{\beta}\}_{\beta\in \Gamma_1\subset \Gamma}$ such that \begin{equation}\label{yo1}
 		|u\square v_\beta(a^*) -u\square v(a^*)|>2\delta~~\forall \beta\in \Gamma_1
 	\end{equation}
 	without loss of generality we can assume that $\|u\|\leq 1, \|v\|\leq 1$, $\|v_\alpha\|\leq 1$ for all $\alpha$ and $\|a^*\|\leq 1$. \\
 	Now let $(\mathcal A, \times)$ be any $\epsilon$-perturbation of $(\mathcal A, \cdot, \|.\|)$ with an equivalent norm, where $\epsilon<\frac{\delta}{2}$. Then due to lemma \ref{perturb}, we have \begin{equation}\label{yo2}
 		|u\square_1v_\beta(a^*)-u\square v_\beta(a^*)|\leq \frac{\delta}{2},~~	|u\square_1v(a^*)-u\square v(a^*)|\leq \frac{\delta}{2}
 	\end{equation}
 	Using the equations \ref{yo1}, $\ref{yo2}$ and invoking triangle inequality, we deduce that 
$$|u\square_1 v_\beta(a^*)-u\square_1v(a^*)|\geq \delta~~~\forall \beta\in \Gamma_1.$$
Since, $\Gamma_1$  is co-final in $\Gamma$, we know that the map $v\mapsto u\square_1 v$ is not weak*-weak* continuous in $(\mathcal A, \times, \|.\|)^{**}$ and hence the algebra $(\mathcal A, \times)$ is not Arens regular.
 \end{proof}
 It is well known that $\epsilon$-perturbations do not preserve commutativity (or non-commutativity). For example consider a non-commutative Banach algebra $(\mathcal A,\cdot, \|.\|)$. For any $k\geq 2$, define two new products $a\times_1 b=\frac{1}{k}a\cdot b$ and $a\times_2 b=\frac{1}{k}(a\cdot b+b\cdot a)\leq $. Clearly $(\mathcal A, \times_1, \|.\|)$ is a Banach algebra and then one can easily notice that $\|a\times_1 b-a\times_2 b\|=\|\frac{1}{k}ba\|\leq \frac{1}{k}\|a\|.\|b\|$, i.e algebra $(\mathcal A, \times_2)$ is the $\frac{1}{k}$-perturbation of Banach algebra $(\mathcal A, \times_1,\|.\|)$. The former is non commutative while the latter is commutative.\\
 
It is known that Arens regularity is generally not preserved under $\epsilon$-perturbations of Banach algebra. But for certain class of Banach algebra, Arens regularity is actually preserved under $\epsilon$-perturbations. One such class of algebras are \emph{Banach function algebras equivalent to uniform algebras}, as we shall prove now.
\begin{definition} \cite[Def. 3.1.2]{Dales}
	A Banach function algebra $\mathcal B$ is a function algebra (point separating subalgebra of $C(S)$ for some locally compact space $S$ ) with a norm $\|.\|$ such that $\|f\|\geq \|f\|_\infty$  and $(\mathcal B, \|.\|)$ is a Banach algebra.  A Banach function algebra $(\mathcal B,\|.\|)$ is said to be equivalent to a uniform algebra if $\|.\|\sim \|.\|_\infty$.
\end{definition}
Banach function algebras are by default commutative and semi-simple (see \cite[Definition 3.1.2]{Dales}). Although the uniform algebra $C(S)$ (w.r.t uniform norm) is always Arens regular, Banach function algebras may sometimes fail to be Arens regular. For example the Fourier algebra of a compact infinite group is a Banach function algebra but is not Arens regular. But Banach function algebras which are equivalent to a uniform algebra are always Arens regular. Hence, we have the following.
\begin{theorem}\label{BFAU}
	$\epsilon$-perturbations of a Banach function algebra equivalent to a uniform algebra are always Arens regular.
\end{theorem}
\begin{proof}
	Notice that Banach function algebras which are equivalent to a uniform algebras are always Arens regular (because uniform algebras are subalgebras (with supremum norm) of the $C^*$-algebra $C(S)$ for some compact space $S$). We shall prove that $\epsilon$-perturbation of Banach function algebra which are equivalent to a uniform algebra, are again equivalent to a uniform algebra.
	Let $(\mathcal B, \|.\|)$ be a Banach function algebra equivalent to a uniform algebra, i.e $\|.\|\sim \|.\|_\infty$. Hence, there exists $c_1,c_2>0$ such that \begin{equation}
		c_1\|f\|_\infty\leq \|f\|\leq c_2\|f\|_\infty~~\forall f\in \mathcal B
	\end{equation}
	Let $(\mathcal B, \times)$ be a $\epsilon$-perturbation of Banach function algebra $\mathcal B$ and $\|.\|_1$ be an equivalent norm such that $(\mathcal B, \times, \|.\|_1)$ becomes a Banach algebra.. Then $\|f.g-f\times g\|\leq \epsilon\|f\|.\|g\|$. Hence,
	\begin{align*}
		\|f.f\|-\epsilon\|f\|^2\leq\|f\times f\|\leq (1+\epsilon)\|f\|^2
	\end{align*}
	Using the equivalence of three norms, we can deduce that 
	\begin{equation}
		k_1\|f\|_1^2\leq \|f\times f\|_1\leq k_2\|f\|_1^2
	\end{equation}
for some $k_1,k_2>0$.	Through induction one can see that the spectral radius $r_1(f)$ of $f$ in the Banach algebra $(\mathcal B, \times ,\|.\|_1)$ satisfies $$k_1\|f\|_1\leq r_1(f)\leq k_2\|f\|_1~~~\forall f\in \mathcal B.$$
Hence, by \cite[Corollary 5 ]{Gerd}, the Banach algebra $(\mathcal B, \times, \|.\|_1)$ is commutative and further by a theorem of Hirschfeld and Zelazko \cite[Corollary 8]{Gerd} the spectral radius $r_1(f)$ is a norm and $(\mathcal B, \times, r_1)$ forms a uniform Banach algebra (because $r_1(f\times f)=(r_1(f))^2$ holds true). Thus, the algebra $(\mathcal B, \times , \|.\|_1)$ is Arens regular (because it is equivalent to a uniform algebra.)
\end{proof}
The above discussion prompt us to define the following.
\begin{definition}
	An Arens regular Banach algebra $\mathcal A$ is said to have stable Arens regularity (SAR) if there exists a $K>0$ such that for every $0<\epsilon\leq K$, the $\epsilon$-perturbation of $\mathcal A$ is Arens regular.
\end{definition}
Clearly, by theorem \ref{BFAU}, Banach function algebras which are equivalent to a uniform algebra have stable Arens regularity (SAR) property. It is evident from the \cite[Th. 1]{Palacios} that every $C^*$-algebra has stable Arens regularity (SAR) property. We end this section with the following natural questions.\\

\noindent Question: Does the Banach algebra $\ell_1(\mathbb Z)$ (with respect to point-wise multiplication) has stable Arens regularity property?\\

\noindent Question: Find Banach algebras with SAR property which are not a Banach function algebra or a $C^*$-algebra?
\section*{Funding} 

This research received institutional support from Amrita Vishwavidyapeetham (Deemed to be University), Haridwar Campus, India, where the author is affiliated to. No specific external grant was received for this work.

\section*{Data Availability Statement}

Data sharing is not applicable to this article as no datasets were generated or analyzed during the current study.
\section*{Conflict of Interest}

The authors declare that they have no known competing financial interests or personal relationships that could have appeared to influence the work reported in this paper.

~\\~\\~\\
 Deepika Rajoriya, \textit{Department of Mathematics, Amrita Vishwavidyapeetham Deemed University, Haridwar Campus, Haridwar, Uttrakhand, India.}\\
Email: \url{deepikarajoriya@outlook.com}
\end{document}